\documentclass[reqno,a4paper]{amsart}
\usepackage{amssymb}
\usepackage{amsmath}

\begin{document} 
\newtheorem{prop}{Proposition}[section]
\newtheorem{Def}{Definition}[section] \newtheorem{theorem}{Theorem}[section]
\newtheorem{lemma}{Lemma}[section] \newtheorem{Cor}{Corollary}[section]

\title[3D Klein - Gordon - Schr\"odinger system]{\bf Low regularity 
well-posedness for the 3D Klein - Gordon - 
Schr\"odinger system}
\author[Hartmut Pecher]{
{\bf Hartmut Pecher}\\
Fachbereich Mathematik und Naturwissenschaften\\
Bergische Universit\"at Wuppertal\\
Gau{\ss}str.  20\\
42097 Wuppertal\\
Germany}
\email{\tt pecher@math.uni-wuppertal.de}
\date{}

\begin{abstract} 
The Klein-Gordon-Schr\"odinger system in 3D is shown to be locally well-posed 
for Schr\"odinger data in $H^s$ and wave data in $H^{\sigma}\times H^{\sigma 
-1}$ , if $ s > - \frac{1}{4},$  $\sigma > - \frac{1}{2}$ , $\sigma -2s > 
\frac{3}{2} $ and $\sigma -2 < s < \sigma +1$ . This result is optimal up to 
the endpoints in the sense that the local flow map is not $C^2$ otherwise. It is 
also shown that (unconditional) uniqueness holds for $s=\sigma=0$ in the natural 
solution space $C^0([0,T],L^2) \times C^0([0,T],L^2) \times 
C^0([0,T],H^{-\frac{1}{2}}).$ This solution exists even globally by 
Colliander, Holmer and Tzirakis \cite{CHT}. The proofs are based on new 
well-posedness results for the Zakharov system by Bejenaru, Herr, Holmer and 
Tataru \cite{BHHT}, and Bejenaru and Herr \cite{BH}.
\end{abstract}
\maketitle

\renewcommand{\thefootnote}{\fnsymbol{footnote}}
\footnotetext{\hspace{-1.8em}{\it 2000 Mathematics Subject Classification:} 
35Q55, 35L70 }

\normalsize 
\setcounter{section}{0}
\section{Introduction and main results}

We consider the (3+1)-dimensional Cauchy problem for the Klein - Gordon - 
Schr\"odinger system with Yukawa coupling
\begin{eqnarray}
\label{0.1}
i\partial_t u + \Delta u & = & nu
\\
\label{0.2}
\partial_t^2 n + (1-\Delta) n & = & |u|^2
\end{eqnarray}
with initial data
\begin{equation}
u(0)  =  u_0 \,  , \, n(0)  =  n_0 \, , \, \partial_t 
n(0) = n_1 \, ,
\label{0.3}
\end{equation}
where $u$ is a complex-valued and $n$ a real-valued function defined for $(x,t) 
\in {\mathbb R}^3 \times [0,T]$ .
This is a classical model which describes a system of scalar nucleons 
interacting with neutral scalar mesons. The nucleons are described by the 
complex scalar field $u$ and the mesons by the real scalar field $n$. The mass 
of the meson is normalized to be 1.

Our results do not use the energy conservation 
law but only charge conservation $ \|u(t)\|_{L^2({\mathbb R}^3)} \equiv const $ 
(for 
the global existence result), so they are equally true if one replaces $nu$ and 
$|u|^2$ by $-nu$ and/or $-|u|^2$ , respectively.

We are interested in local and global solutions for data
$$ u_0 \in H^s({\mathbb R}^3) \, , \, n_0 \in H^{\sigma}({\mathbb R}^3) \, , \, 
n_1 \in 
H^{\sigma -1}({\mathbb R}^3) $$
with minimal $s$ and $\sigma$ .

Local well-posedness for data $ u_0 \in L^2({\mathbb R}^3) \, , \, n_0 \in 
L^2({\bf 
R}^3) \, , \, n_1 \in H^{-1}({\mathbb R}^3) $ was shown by the author \cite{P} 
based on 
estimates given by Ginibre, Tsutsumi and Velo \cite{GTV} for the Zakharov 
system, more precisely these solutions exist uniquely in Bourgain type spaces 
which are subsets of the natural solution spaces
\begin{equation}
\label{*}
(u,n,\partial_t n) \in C^0([0,T],L^2({\mathbb R}^3)) \times C^0([0,T],L^2({\bf 
R}^3)) \times C^0([0,T],H^{-1}({\mathbb R}^3)) \, . 
\end{equation}
Colliander, Holmer and Tzirakis \cite{CHT} proved that under the same 
assumptions on the data global well-posedness holds in Strichartz-type spaces, 
especially (\ref{*}) holds for any $T > 0$.

Concerning the closely related wave Schr\"odinger system local well-posedness 
was shown for $ s > - \frac{1}{4} $ and $ \sigma > - \frac{1}{38} $ and also 
global well-posedness for certain $ s,\sigma < 0$ by T. Akahori 
\cite{A},\cite{A1}.

Thus three questions arise:
\\
(a)
Can we show local well-posedness for even rougher data ?
\\
(b) Is it possible to show the sharpness of this local well-posedness result 
?
\\
(c) Under which assumptions on the data can we show unconditional uniqueness, 
i.e. uniqueness in the natural solution space ?

Concerning (a) we prove that local well-posedness holds in Bourgain type spaces 
which are subsets of the natural spaces, provided the data fulfill 
$$ s > - \frac{1}{4} \, , \, \sigma >- \frac{1}{2} \, , \, \sigma -2s < 
\frac{3}{2} \, , \, \sigma -2 < s < \sigma +1 \, . $$
Especially the choice $ s = - \frac{1}{4} + $ and $ \sigma = - \frac{1}{2}+$ is 
possible, thus we relax the regularity assumptions for the Schr\"odinger and 
wave parameters by almost $\frac{1}{4}$ and $\frac{1}{2}$ order of derivatives, 
respectively.

Concerning (b) the estimates for the nonlinearities which lead to the local 
well-posedness result fail if $s < - \frac{1}{4} $ or $ \sigma < - \frac{1}{2}$ 
or $ \sigma -2s > \frac{3}{2} $. Using ideas of Holmer \cite{H} and Bejenaru, 
Herr, Holmer and Tataru \cite{BHHT} we can even show that the solution map
$(u_0,n_0,n_1) \mapsto (u(t),n(t),\partial_t n(t))$ is not $C^2$ in these 
cases, i.e. some type of ill-posedness holds.  

Concerning (c) we show that for data $ u_0 \in L^2({\mathbb R}^3) \, , \, n_0 
\in 
L^2({\mathbb R}^3) \, , \, n_1 \in H^{-1}({\mathbb R}^3) $ unconditional 
uniqueness 
holds in the space (\ref{*}). Using the global existence result of \cite{CHT} 
we get unconditional global well-posedness in this case.

The question of unconditional uniqueness was considered among others by Yi Zhou 
for the KdV equation \cite{Z} and nonlinear wave equations \cite{Z1}, by N. 
Masmoudi and K. Nakanishi for the Maxwell-Dirac, the 
Maxwell-Klein-Gordon equations \cite{MN}, the Klein-Gordon-Zakharov system and 
the Zakharov system \cite{MN1}, and by F. Planchon \cite{Pl} for semilinear wave 
equations.

The results in this paper are based on the (3+1)-dimensional estimates by 
Bejenaru and Herr \cite{BH} which they recently used to show a sharp 
well-posedness result for the Zakharov system. We also use the corresponding 
sharp (2+1)-dimensional local well-posedness results for the Zakharov system by 
Bejenaru, Herr, Holmer and Tataru \cite{BHHT}, especially their 
counterexamples.

We use the standard Bourgain spaces $X^{m,b}$ for the Schr\"odinger equation, 
which are defined as the completion of ${\mathcal S}({\mathbb R}^3 \times 
{\mathbb R})$ 
with respect to
$$ \|f\|_{X^{m,b}} := \| \langle \xi \rangle ^m \langle \tau + |\xi|^2 \rangle 
^b \widehat{f}(\xi,\tau)\|_{L^2_{\xi\tau}} \, . $$
Similarly $X^{m,b}_{\pm}$ for the equation $i\partial_t n_{\pm} \mp 
A^{1/2}n_{\pm} = 0$ is the completion of ${\mathcal S}({\mathbb R}^3 \times 
{\mathbb R})$ 
with respect to
$$ \|f\|_{X^{m,b}_{\pm}} := \| \langle \xi \rangle ^m \langle \tau \pm |\xi| 
\rangle ^b \widehat{f}(\xi,\tau)\|_{L^2_{\xi\tau}} \, . $$
For a given time interval $I$ we define $ \|f\|_{X^{m,b}(I)} := 
\inf_{\tilde{f}_{|I} = f} \|\tilde{f}\|_{X^{m,b}} $ and similarly $ 
\|f\|_{X^{m,b}_{\pm}(I)}$ . We often skip $I$ from the notation.

In the following we mean by a solution of a system of differential equation
always a solution of the corresponding system of integral equations.

Before formulating the main results of our paper we recall that the KGS system 
can be transformed into a first order (in t) system as follows: if
$$ (u,n,\partial_t n) \in C^0([0,T],H^s) \times  C^0([0,T],H^{\sigma}) \times  
C^0([0,T],H^{\sigma -1}) $$ is a solution of (\ref{0.1}),(\ref{0.2}),(\ref{0.3}) 
with data $(u_0,n_0,n_1) \in H^s \times H^{\sigma} \times H^{\sigma -1}$ ,then 
defining  $A:= - 
\Delta + 1$ and 
$$ n_{\pm} := n \pm iA^{-\frac{1}{2}} \partial_t n  $$ 
and 
$$ n_{\pm 0} := n_0 \pm iA^{-\frac{1}{2}}n_1 \in H^{\sigma} \, ,$$
 we get that
$$(u,n_+,n_-) \in C^0([0,T],H^s) \times  C^0([0,T],H^{\sigma}) \times  
C^0([0,T],H^{\sigma})$$
is a solution of the following problem:
\begin{eqnarray}
\label{0.1'}
i \partial_t u + \Delta u & = & \frac{1}{2}(n_+ - n_-)u \\
\label{0.2'}
i \partial_t n_{\pm} \mp A^{1/2} n_{\pm} & = & \pm A^{-1/2}(|u|^2) \\
\label{0.3'}
u(0) = u_0 & , & n_{\pm}(0) = n_{\pm 0} := n_0 \pm i A^{-1/2}n_1 \, . 
\end{eqnarray}
The corresponding system of integral equations reads as follows: 
\begin{eqnarray}
u(t) & = & e^{it\Delta} u_0 + \frac{1}{2} \int_0^t e^{i(t-\tau)\Delta} 
(n_+(\tau)+n_-(\tau))u(\tau) d\tau \\
\label{9}
n_{\pm}(t) & = & e^{\mp itA^{1/2}}n_{\pm 0} \pm i \int_0^t e^{\mp 
i(t-\tau)A^{1/2}} 
A^{-1/2}(|u(\tau)|^2) d\tau \, .
\end{eqnarray}

Conversely, if 
$$(u,n_+,n_-) \in X^{s,b}[0,T] \times X_+^{\sigma,b}[0,T]  \times 
X_-^{\sigma,b}[0,T]$$
is a solution of (\ref{0.1'}),(\ref{0.2'}) with data $u(0)=u_0 \in H^s$ and 
$n_{\pm}(0) = n_{\pm 0} \in H^{\sigma}$ , then we define
$n := \frac{1}{2}(n_+ + n_-)$ , $ 2i A^{-\frac{1}{2}}\partial_t n := n_+ - n_-$ 
and conclude that
$$(u,n,\partial_t n) \in X^{s,b}[0,T] \times (X_+^{\sigma,b}[0,T] + 
X_-^{\sigma,b}[0,T]) \times (X_+^{\sigma -1,b}[0,T] + X_-^{\sigma -1,b}[0,T]) $$
is a solution of (\ref{0.1}),(\ref{0.2}) with data 
$ u(0)=u_0 \in H^s$ and
 $$ n(0)=n_0= \frac{1}{2}(n_+(0)+n_-(0)) \in H^{\sigma} \, , \, \partial_t n(0) 
= \frac{1}{2i} A^{\frac{1}{2}}(n_+(0)-n_-(0)) \in H^{\sigma -1} \, . $$
If $(u,n_+,n_-) \in C^0([0,T],H^s) \times  C^0([0,T],H^{\sigma}) \times  
C^0([0,T],H^{\sigma}) $ , then we also have $(u,n,\partial_t n) \in 
C^0([0,T],H^s) \times  C^0([0,T],H^{\sigma}) \times  C^0([0,T],H^{\sigma -1})$ .

Our local well-posedness result reads as follow:
\begin{theorem}
\label{Theorem 1}
The Cauchy problem for the Klein - Gordon - Schr\"odinger system 
(\ref{0.1}),(\ref{0.2}),(\ref{0.3}) is locally well-posed for data 
$$ u_0 \in H^s({\mathbb R}^3) \, , \, n_0 \in H^{\sigma}({\mathbb R}^3) \, , \,  
n_1 
\in H^{\sigma -1}({\mathbb R}^3)$$
under the assumptions
$$ s > - \frac{1}{4} \, , \, \sigma > - \frac{1}{2} \, , \, \sigma -2s < 
\frac{3}{2} \, , \, \sigma -2 < s < \sigma +1 \, . $$
More precisley, there exists $ T > 0 $ , 
$T=T(\|u_0\|_{H^s},\|n_0\|_{H^{\sigma}},\|n_1\|_{H^{\sigma -1}})$ and a unique 
solution
$$ u \in X^{s,\frac{1}{2}+}[0,T] \, , $$
$$ n \in X_+^{\sigma,\frac{1}{2}+}[0,T] + 
X_-^{\sigma,\frac{1}{2}+}[0,T] \, , \, \partial_t n  \in X_+^{\sigma 
-1,\frac{1}{2}+}[0,T] + X_-^{\sigma -1,\frac{1}{2}+}[0,T] \, . $$
This solution has the property
$$ u \in C^0([0,T],H^s({\mathbb R}^3)) \, , \,  n \in C^0([0,T],H^{\sigma}({\bf 
R}^3)) \, , \, \partial_t n \in C^0([0,T],H^{\sigma -1}({\mathbb R}^3)) \, . $$
\end{theorem}

These conditions are sharp up to the endpoints, more precisely we get
\begin{theorem}
\label{Theorem 2}
Let $ u_0 \in H^s({\mathbb R}^3)$ ,  $ n_0 \in H^{\sigma}({\mathbb R}^3)$ , $ 
n_1 \in 
H^{\sigma -1}({\mathbb R}^3)$ . Then the flow map $ (u_0,n_0,n_1) \mapsto 
(u(t),n(t),\partial_t n(t)) $ , $ t \in [0,T] $ , does not belong to $C^2$ for 
any $ T > 0 $ , provided $ \sigma -2s-\frac{3}{2} > 0 $ or $ s < - \frac{1}{4} $ 
or $ 
\sigma < - \frac{1}{2} $ .
\end{theorem}

The necessary estimates for the nonlinearities required in the local existence 
results are false if the assumptions regarding the smoothness of the data are 
violated. This is proven in section 4, Prop. \ref{Prop. 2} and Prop. \ref{Prop. 
3}. 

The unconditional uniqueness result is the following:
\begin{theorem}
\label{Theorem 3}
Let
$ u_0 \in L^2({\mathbb R}^3) \, , n_0 \in L^2({\mathbb R}^3) \, ,n_1 \in 
H^{-1}({\bf 
R}^3)$ be given. The Klein - Gordon - Schr\"odinger system 
(\ref{0.1}),(\ref{0.2}),(\ref{0.3}) is unconditionally globally well-posed, 
i.e. there exists a solution unique in
$$ (u,n,\partial_t n) \in C^0({\mathbb R}^+,L^2({\mathbb R}^3)) \times 
C^0({\mathbb 
R}^+,L^2({\mathbb R}^3)) \times C^0({\mathbb R}^+,H^{-1}({\mathbb R}^3)) \, . $$
\end{theorem}

We use the following notation. The Fourier transform is denoted by  $ \, 
\widehat{}$ or ${\mathcal F}$ and its inverse by $\, \check{} \,$ or ${\mathcal 
F}^{-1}$ , where it should be clear from the context, whether it is taken with 
respect to the space and time variables simultaneously or only with respect to 
the space variables. For real numbers $a$ we denote by $a+$ and $a-$ a number 
sufficiently close to $a$, but
larger and smaller than $a$, respectively. \\
{\bf Acknowlegment:} The author thanks the referee for many valuable suggestions 
which improved the paper.

\section{(Conditional) local well-posedness}
\begin{theorem}
\label{Theorem 4}
Assume 
$ \frac{1}{4} < b_1 \le \frac{1}{2} \, , \, b,b_2 \ge \frac{1}{2} \, , \, s > -1 
\, , $
\begin{eqnarray}
\label{10} 
\sigma & > & \frac{1}{2} -2b_1 \, , \\
\label{11} 
 s+\sigma &>& -2b_2 \, , \\
\label{12} 
s & < & \sigma + 2b_1 \, .
\end{eqnarray} 
If $ 0 < T \le 1 $ and $ u \in X^{s,b_2}[0,T] \, , \, v \in 
X^{\sigma,b}_{\pm}[0,T]$ we have
$$ \|uv\|_{X^{s,-b_1}[0,T]} \le c \|u\|_{X^{s,b_2}[0,T]} 
\|v\|_{X_{\pm}^{\sigma,b}[0,T]} \, . $$
\end{theorem}
\noindent {\bf Remark:} A possible choice of the parameters is:
$$ b_1 = \frac{1}{2}- \, , \, b=b_2= \frac{1}{2} \, , \, \sigma > - \frac{1}{2} 
\, , \, s > - \frac{1}{2} \, , \, s < \sigma +1 \, . $$

Because we are going to use dyadic decompositions of $\widehat{u}$ and 
$\widehat{v}$ we take the notation from \cite{BH} and start by choosing a 
function $\psi \in C^{\infty}_0((-2,2))$ , which is even and  nonnegative with 
$\psi(r) = 1 $ for $|r|\le 1$. Defining $\psi_N(r) = \psi(\frac{r}{N}) - 
\psi(\frac{2r}{N}) $ for dyadic numbers $N=2^n \ge 2$ and $\psi_1 = \psi$ we 
have $ 1 = \sum_{N \ge 1} \psi_N $ . Thus $supp \, \psi_1 \subset[-2,2]$ and 
$supp \,
\psi_N \subset [-2N,-N/2] \cup [N/2,2N] $ for $N \ge 2$. For $f:{\mathbb R}^3 
\to 
{\mathbb C}$ we define the dyadic frequency localization operators $P_N$ by
$$ {\mathcal F}_x(P_N f)(\xi) = \psi_N(|\xi|) {\mathcal F}_x f(\xi) \, . $$
For $u:{\mathbb R}^3 \times {\mathbb R} \to {\mathbb C}$ we define the 
modulation 
localization operators
\begin{eqnarray*}
{\mathcal F}(S_Lu)(\tau,\xi) & = & \psi_L(\tau + |\xi|^2) {\mathcal 
F}u(\tau,\xi) \\
{\mathcal F}(W_L^{\pm} u)(\tau,\xi) & = & \psi_L(\tau \pm |\xi|) {\mathcal 
F}u(\tau,\xi)
\end{eqnarray*}
in the Schr\"odinger case and the wave case.
\begin{proof} [Proof of Theorem \ref{Theorem 4}]  Defining
$$ I(f,g_1,g_2) = \int f(\xi_1 - \xi_2,\tau_1 - \tau_2) g_1(\xi_1,\tau_1) 
g_2(\xi_2,\tau_2) d\xi_1 d\xi_2 d\tau_1 d\tau_2 $$
we have to show
$$|I(\widehat{v},\widehat{u}_1,\widehat{u}_2)| \lesssim 
\|u_1\|_{X^{-s,b_1}} 
\|u_2\|_{X^{s,b_2}}\|v\|_{X_{\pm}^{\sigma,b}} \, .
 $$
We use dyadic decompositions
$$ u_k = \sum_{N_k,L_k \ge 1} S_{L_k} P_{N_k} u_k \, , \, v = \sum_{N,L \ge 1} 
W^{\pm}_L P_N v \, . $$
Defining
$$ g_k^{L_k,N_k} = {\mathcal F} S_{L_k} P_{N_k} u_k \, , \, f^{L,N} = {\mathcal 
F} 
W^{\pm}_L P_N v $$
we have
$$ I(\widehat{v},\widehat{u}_1,\widehat{u}_2) = \sum_{N,N_1,N_2 \ge 1 , 
L,L_1,L_2 \ge 1} I(f^{L,N},g_1^{L_1,N_1},g_2^{L_2,N_2}) \, . $$
{\bf Case 1:} $ N_1 \sim N_2 \gtrsim N \gg 1 $ \\
{\bf a.} In the case $L,L_1,L_2 \le N_1^2$ we get by \cite[formula 
(3.24)]{BH}:
\begin{eqnarray*}
\lefteqn{ |I(f^{L,N},g_1^{L_1,N_1},g_2^{L_2,N_2})| } \\
& \lesssim & N_1^{-\frac{1}{2}+} L^{\frac{1}{2}} \|f^{L,N}\|_{L^2} 
L_1^{\frac{1}{2}} \|g_1^{L_1,N_1}\|_{L^2} L_2^{\frac{1}{2}} 
\|g_2^{L_2,N_2}\|_{L^2} \\
& \lesssim & A N_1^{-\frac{1}{2}+s+} N_2^{-s+} N^{-\sigma +} 
L_1^{\frac{1}{2}-b_1+} L_2^{\frac{1}{2}-b_2+} L^{\frac{1}{2}-b+} \\
& \lesssim & A N_1^{\frac{1}{2}-2b_1+} N^{-\sigma +} \\
& \lesssim & A \, ,
\end{eqnarray*}
because $b_1 > \frac{1}{4}$ and (\ref{10}). Here
$$A:= N^{\sigma -} L^{b-} \|f^{L,N}\|_{L^2}  N_1^{-s -} L_1^{b_1-} 
\|g_1^{L_1,N_1}\|_{L^2} N_2^{s -} L_2^{b_2-} \|g_2^{L_2,N_2}\|_{L^2} \, . $$
Dyadic summation over $L,L_1,L_2 \ge 1$ and $N,N_1,N_2$ gives
\begin{eqnarray*}
\lefteqn{ \sum_{1 \ll N \lesssim N_1 \sim N_2} \sum_{1 \le L,L_1,L_2 \le N_1^2} 
|I(f^{L,N},g_1^{L_1,N_1},g_2^{L_2,N2})| } \\
& \lesssim & \Bigl( \sum_{N,L} (N^{\sigma} L^b \|f^{L,N}\|_{L^2})^2 
\Bigr)^{\frac{1}{2}} \Bigl( \sum_{N_1,L_1} (N_1^{-s} L_1^{b_1} 
\|g_1^{L_1,N_1}\|_{L^2})^2 \Bigr)^{\frac{1}{2}} \\ 
& & \Bigl( \sum_{N_2,L_2} (N_2^s L_2^{b_2} \|g_2^{L_2,N_2}\|_{L^2})^2 
\Bigr)^{\frac{1}{2}} \\
& \lesssim & \|v\|_{X_{\pm}^{\sigma,b}} \|u_1\|_{X^{-s,b_1}} \|u_2\|_{X^{s,b_2}} 
\end{eqnarray*}
by almost orthogonality. \\
{\bf b.} Similarly in the case $N_1^2 < \max(L,L_1,L_2)$ we get by 
\cite[formula (3.25)]{BH}:
\begin{eqnarray*}
\lefteqn{ |I(f^{L,N},g_1^{L_1,N_1},g_2^{L_2,N_2})| } \\
& \lesssim & N_1^{-\frac{1}{2}} \frac{N_1}{\max(L,L_1,L_2)^{\frac{1}{2}}} 
L^{\frac{1}{2}} \|f^{L,N}\|_{L^2} L_1^{\frac{1}{2}} \|g_1^{L_1,N_1}\|_{L^2} 
L_2^{\frac{1}{2}} \|g_2^{L_2,N_2}\|_{L^2} \\
& \lesssim & A N_1^{-\frac{1}{2}+s+} N_2^{-s+} N^{-\sigma+} 
L_1^{\frac{1}{2}-b_1+} L_2^{\frac{1}{2}-b_2+} L^{\frac{1}{2}-b+} 
\frac{N_1}{\max(L,L_1,L_2)^{\frac{1}{2}}} \\
& \lesssim & A N_1^{\frac{1}{2}+} N^{-\sigma +} \frac{1}{\max(L,L_1,L_2)^{b_1 
-}} \\
& \lesssim & A N_1^{\frac{1}{2}-2b_1+} N^{-\sigma +} \\
& \lesssim & A
\end{eqnarray*}
as in case a. Dyadic summation gives the claimed estimate.
\\
{\bf Case 2:} $ N_1 \ll N_2 $ ($ \Longrightarrow \, N \sim N_2$) \\
{\bf a.} In the case $L_2 \ll N_2^2 $ we get by \cite[formula (3.26) and 
(3.28)]{BH} :
\\
$ \max(L,L_1,L_2) \gtrsim N_2^2 $
and
\begin{eqnarray*}
\lefteqn{ |I(f^{L,N},g_1^{L_1,N_1},g_2^{L_2,N_2})| } \\
& \lesssim & N_1 N_2^{-\frac{1}{2}} (L_1 L_2 L)^{\frac{1}{2}} 
\max(L,L_1,L_2)^{-\frac{1}{2}} \|f^{L,N}\|_{L^2} \|g_1^{L_1,N_1}\|_{L^2} 
\|g_2^{L_2,N_2}\|_{L^2} \\
& \lesssim & A N_1^{1+s+} N_2^{-\frac{1}{2}-s+} N^{-\sigma+} 
L_1^{\frac{1}{2}-b_1+} L_2^{\frac{1}{2}-b_2+} L^{\frac{1}{2}-b+} 
\max(L,L_1,L_2)^{-\frac{1}{2}} \\
& \lesssim & A N_2^{\frac{1}{2}-\sigma+} \max(L,L_1,L_2)^{-b_1+}\\
& \lesssim & A N_2^{\frac{1}{2}-2b_1-\sigma+} \\
& \lesssim & A 
\end{eqnarray*}
using $s > -1$ and (\ref{10}).
Dyadic summation gives the desired bound as in case 1.\\
{\bf b.} In the case $L_2 \gtrsim N_2^2$ we consider 3 subcases using the proof 
of 
\cite[Prop. 3.8]{BH}:  \\
{\bf b1.} $L \le L_1$ and $N_1^2 \le \max(L,L_1) = L_1 $. 
\begin{eqnarray*}
 |I(f^{L,N},g_1^{L_1,N_1},g_2^{L_2,N_2})| 
& \lesssim & N_1^{\frac{3}{2}} L^{\frac{1}{2}} \|f^{L,N}\|_{L^2} 
\|g_1^{L_1,N_1}\|_{L^2} 
\|g_2^{L_2,N_2}\|_{L^2} \\
& \lesssim & A N_1^{\frac{3}{2}+s+} N_2^{-s+} N^{-\sigma +} L_1^{-b_1 +} 
L_2^{-b_2 +} L^{\frac{1}{2}-b +} \\
& \lesssim & A N_1^{\frac{3}{2}+s+} N_2^{-s-\sigma +} L_1^{-b_1 +} L_2^{-b_2 +}  
\\
& \lesssim & A N_1^{\frac{3}{2}+s-2b_1 +} N_2^{-s-\sigma -2b_2 +}   \\
& \lesssim & A N_1^{\frac{3}{2}-\sigma -2b_1 -2b_2 +}    \\
& \lesssim & A
\end{eqnarray*}
using (\ref{11}),(\ref{10}) and $b_2 \ge \frac{1}{2}$ . Dyadic 
summation gives the claimed estimate.\\
{\bf b2.} $ L_1 < L $ and $ N_1^2 \le \max(L,L_1) = L $ . \\
We have
\begin{eqnarray*}
|I(f^{L,N},g_1^{L_1,N_1},g_2^{L_2,N_2})|   
& \lesssim  & N_1^{\frac{3}{2}} L_1^{\frac{1}{2}} \|f^{L,N}\|_{L^2} 
\|g_1^{L_1,N_1}\|_{L^2} \|g_2^{L_2,N_2}\|_{L^2} \\ 
& \lesssim & A N_1^{\frac{3}{2}+s+} N_2^{-s+} N^{-\sigma +} L_1^{\frac{1}{2}-b_1 
+} L_2^{-b_2 +} L^{-b+} \\ 
& \lesssim & A N_1^{\frac{3}{2}+s+} N_2^{-s-\sigma +} L^{\frac{1}{2}-b_1-b +} 
L_2^{-b_2 +} \\
& \lesssim & A N_1^{\frac{3}{2}+s+} N_2^{-s-\sigma +} N_1^{1-2b_1-2b +} 
N_2^{-2b_2 +} \\
& \lesssim & A N_1^{\frac{3}{2}+s+1-2b_1-2b+} N_2^{-s-\sigma -2b_2 +}  \\
& \lesssim & A N_1^{\frac{5}{2}-\sigma -2b_1 -2b_2 -2b+}   \\
& \lesssim & A
\end{eqnarray*} 
again using (\ref{11}),(\ref{10}) and $b,b_2 \ge \frac{1}{2}$.
\\
{\bf b3.} $ N_1^2 > \max(L,L_1) $ . \\
We have
\begin{eqnarray*}
|I(f^{L,N},g_1^{L_1,N_1},g_2^{L_2,N_2})|
& \lesssim  & N_1^{\frac{1}{2}} L^{\frac{1}{2}} L_1^{\frac{1}{2}} 
\|f^{L,N}\|_{L^2} 
\|g_1^{L_1,N_1}\|_{L^2} \|g_2^{L_2,N_2}\|_{L^2} \\ 
& \lesssim & A N_1^{\frac{1}{2}+s+} N_2^{-s+} N^{-\sigma +} 
L_1^{\frac{1}{2}-b_1+} L_2^{-b_2+} L^{\frac{1}{2}-b+} \\
& \lesssim & A
N_1^{\frac{1}{2}+s+} 
N_2^{-s-\sigma+} N_1^{1-2b_1+} N_2^{-2b_2+} \\
&  \lesssim & A N_1^{\frac{3}{2}-\sigma-2b_1-2b_2+}\\ 
& \lesssim & A 
\end{eqnarray*}
again using (\ref{11}),(\ref{10}) and $b_2 \ge \frac{1}{2}$ .\\
{\bf Case 3:} $ N_2 \ll N_1 $ ($\Longrightarrow$ $N \sim N_1$). \\
Interchanging the roles of $N_1$ and $N_2$ as well as $L_1$ and $L_2$ we 
consider different cases. \\
{\bf a.} $L_1 \ll N_1^2$. \\
We use \cite[formula (3.26) and (3.28)]{BH} and get $\max(L,L_1,L_2) \gtrsim 
N_1^2$ and
\begin{eqnarray*}
\lefteqn{ |I(f^{L,N},g_1^{L_1,N_1},g_2^{L_2,N_2})| } \\ 
 & \lesssim  &  N_2 N_1^{-\frac{1}{2}} (L_1 L_2 L)^{\frac{1}{2}} 
\max(L,L_1,L_2)^{-\frac{1}{2}} \|f^{L,N}\|_{L^2} \|g_1^{L_1,N_1}\|_{L^2} 
\|g_2^{L_2,N_2}\|_{L^2} \\
& \lesssim & A N_1^{s-\frac{1}{2}+} N_2^{-s+1+} N^{-\sigma +} 
L_1^{\frac{1}{2}-b_1+} L_2^{\frac{1}{2}-b_2+} L^{\frac{1}{2}-b+} 
\max(L,L_1,L_2)^{-\frac{1}{2}}\\
& \lesssim & A N_1^{s-\frac{1}{2}-\sigma +} N_2^{-s+1+} \max(L,L_1,L_2)^{-b_1+} 
\\
& \lesssim &  A N_1^{s-\frac{1}{2}-\sigma-2b_1 +} N_2^{-s+1+} \, .
 \end{eqnarray*} 
If $s\le 1$ we get the bound $  A N_1^{\frac{1}{2}-\sigma-2b_1} \lesssim A $ by 
(\ref{10}), whereas for
$s>1$ we estimate by $AN_1^{s-\frac{1}{2}-\sigma-2b_1 +} \lesssim 
AN_1^{-\frac{1}{2}} \lesssim A$ using (\ref{12}). 
Dyadic summation gives the claimed estimate.\\
{\bf b.} In the case $L_1 \gtrsim N_1^2$ we consider 3 subcases using the proof 
of \cite[Prop. 3.8]{BH}
{\bf b1.} $L \le L_2$ , $ N_2^2 \le \max(L,L_2) = L_2 $ . \\
We get
\begin{eqnarray*} 
|I(f^{L,N},g_1^{L_1,N_1},g_2^{L_2,N_2})|  
& \lesssim &  N_2^{\frac{3}{2}} L^{\frac{1}{2}} \|f^{L,N}\|_{L^2} 
\|g_1^{L_1,N_1}\|_{L^2} \|g_2^{L_2,N_2}\|_{L^2} 
\\
& \lesssim & A N_1^{s+} N_2^{-s+\frac{3}{2}+} N^{-\sigma+} L_1^{-b_1+} 
L_2^{-b_2+} L^{\frac{1}{2}-b+} \\
& \lesssim & A N_1^{s-\sigma-2b_1+} N_2^{-s+\frac{3}{2}-2b_2+} \, .
\end{eqnarray*} 
If $s < \frac{3}{2}-2b_2$ we get the bound $AN_1^{\frac{3}{2}-\sigma-2b_1-2b_2+} 
\lesssim AN_1^{1-2b_2} \lesssim A $ by (\ref{10}) and $b_2 \ge \frac{1}{2}$ , 
whereas in the case $ s \ge \frac{3}{2}-2b_2$ we estimate by 
$AN_1^{s-\sigma-2b_1+} \lesssim A$ using (\ref{12}).
\\
{\bf b2.} $L_2 < L$ , $N_2^2 \le \max(L,L_2) = L$ . \\
We get
\begin{eqnarray*} 
|I(f^{L,N},g_1^{L_1,N_1},g_2^{L_2,N_2})|  
& \lesssim &  N_2^{\frac{3}{2}} L_2^{\frac{1}{2}} \|f^{L,N}\|_{L^2} 
\|g_1^{L_1,N_1}\|_{L^2} \|g_2^{L_2,N_2}\|_{L^2} 
\\
& \lesssim & A N_1^{s+} N_2^{-s+\frac{3}{2}+} N^{-\sigma+} L_1^{-b_1+} 
L_2^{\frac{1}{2}-b_2+} L^{-b+} \\
& \lesssim & A N_1^{s-\sigma-2b_1+} N_2^{-s+\frac{3}{2}-2b+} \, .
\end{eqnarray*} 
If $s < \frac{3}{2}-2b$ we get the bound $AN_1^{\frac{3}{2}-\sigma-2b_1-2b+} 
\lesssim AN_1^{1-2b} \lesssim A $ by (\ref{10}) and $b \ge \frac{1}{2}$,  
whereas in the case $ s \ge \frac{3}{2}-2b$ we estimate by 
$AN_1^{s-\sigma-2b_1+} \lesssim A$ using (\ref{12}).
\\
{\bf b3.} $N_2^2 > \max(L,L_2)$ . \\
We get
\begin{eqnarray*} 
|I(f^{L,N},g_1^{L_1,N_1},g_2^{L_2,N_2})|  
& \lesssim &  N_2^{\frac{1}{2}} L^{\frac{1}{2}} L_2^{\frac{1}{2}} 
\|f^{L,N}\|_{L^2} 
\|g_1^{L_1,N_1}\|_{L^2} \|g_2^{L_2,N_2}\|_{L^2} 
\\
& \lesssim & A N_1^{s+} N_2^{-s+\frac{1}{2}+} N^{-\sigma+} L_1^{-b_1+} 
L_2^{\frac{1}{2}-b_2+} L^{\frac{1}{2}-b+} \\
& \lesssim & A N_1^{s-\sigma-2b_1+} N_2^{-s+\frac{1}{2}+} \, .
\end{eqnarray*} 
If $s \le \frac{1}{2}$ we get the bound $AN_1^{\frac{1}{2}-\sigma-2b_1+} 
\lesssim A $ by (\ref{10}) , 
whereas in the case $ s > \frac{1}{2}$ we estimate by 
$AN_1^{s-\sigma-2b_1+} \lesssim A$ using (\ref{12}), which gives the desired 
bound after dyadic summation.
\\
{\bf Case 4:} $N \lesssim 1$ \\
In this case we need no dyadic decomposition. We estimate directly using  $ 
\langle \xi_1 \rangle \sim \langle \xi_2 \rangle$ and $ \xi =\xi_1 - \xi_2 $ , 
$\tau = \tau_1 - \tau_2$ :
\begin{eqnarray*}
\lefteqn{\int_{\langle \xi \rangle \lesssim 1 , \langle \xi_1 \rangle \sim  
\langle \xi_2 \rangle } \widehat{v}(\xi,\tau) \widehat{u}_1(\xi_1,\tau_1) 
\widehat{u}_2(\xi_2,\tau_2)  d\xi_1 d\xi_2 d\tau_1 d\tau_2 } \\
% \end{eqnarray*} 
% \begin{eqnarray*}
& \lesssim & \int_{\langle \xi \rangle \lesssim 1 , \langle \xi_1 \rangle \sim  
\langle \xi_2 \rangle } \frac{\widehat{v}(\xi,\tau)}{\langle \xi \rangle 
^{\frac{3}{2}-\sigma+}} \frac{\widehat{u}_1(\xi_1,\tau_1)}{\langle \xi_1 
\rangle^{-s}}
 \frac{\widehat{u}_2(\xi_2,\tau_2)}{\langle \xi_2 \rangle^s}  d\xi_1 d\xi_2 
d\tau_1 d\tau_2 \\
% \end{eqnarray*} 
% \begin{eqnarray*}
& \lesssim & \|{\mathcal F}^{-1}(\frac{\widehat{v}(\xi,\tau)}{\langle \xi 
\rangle^{\frac{3}{2}-\sigma+}})\|_{L^3_t L^\infty_x}
 \|{\mathcal F}^{-1}(\frac{\widehat{u_1}(\xi_1,\tau_1)}{\langle \xi_1 
\rangle^{-s}})\|_{L^3_t L^2_x}\|{\mathcal 
F}^{-1}(\frac{\widehat{u_2}(\xi_2,\tau_2)}{\langle \xi_2 \rangle^s})\|_{L^3_t 
L^2_x} \\
% \end{eqnarray*} 
% \begin{eqnarray*}
&  \lesssim & \|{\mathcal F}^{-1}(\frac{\widehat{v}(\xi,\tau)}{\langle \xi 
\rangle^{-\sigma}})\|_{L^3_t L^2_x} \|u_1\|_{X^{-s,\frac{1}{6}+}} 
\|u_2\|_{X^{s,\frac{1}{6}+}} \\
% \end{eqnarray*} 
% \begin{eqnarray*}
& \lesssim & \|v\|_{X^{\sigma,\frac{1}{6}+}_{\pm}} \|u_1\|_{X^{-s,\frac{1}{6}+}} 
\|u_2\|_{X^{s,\frac{1}{6}+}} \, .
\end{eqnarray*}
by Sobolev's embedding theorem. This is more than enough for our claimed 
estimate and completes the proof of Theorem \ref{Theorem 4}.
\end{proof}

\begin{theorem}
\label{Theorem 5}
Assume $ s > - \frac{1}{4} \, , \, \sigma-2s < \frac{3}{2} \, , \, \sigma -2 < 
s \, . $ If $0 < T \le 1$ and $u_1,u_2 \in X^{s,\frac{1}{2}-}[0,T]$ the 
following 
estimate holds:
$$ \| A^{-1/2} (u_1 \overline{u}_2) \|_{X^{\sigma,-\frac{1}{2}+}_{\pm}[0,T]} 
\le c 
\|u_1\|_{X^{s,\frac{1}{2}-}[0,T]} \|u_2\|_{X^{s,\frac{1}{2}-}[0,T]} \, . $$
\end{theorem}
\begin{proof}  We have to show
$$|I(\widehat{v},\widehat{u}_1,\widehat{u}_2)| \lesssim  
\|u_1\|_{X^{s,\frac{1}{2}-}} 
\|u_2\|_{X^{s,\frac{1}{2}-}}\|v\|_{X_{\pm}^{1-\sigma,\frac{1}{2}-}} \, .
 $$
Using dyadic decompositions as in the proof of Theorem \ref{Theorem 4} we 
consider different 
cases.\\
{\bf Case 1:} $N_1 \sim N_2 \gtrsim N \gg 1 $\\
{\bf a.} In the case $L,L_1,L_2 \le N_1^2$ we get by \cite[formula (3.24)]{BH} 
\begin{eqnarray*}
\lefteqn{ |I(f^{L,N},g_1^{L_1,N_1},g_2^{L_2,N_2})| } \\
& \lesssim & N_1^{-\frac{1}{2}+} L^{\frac{1}{2}} \|f^{L,N}\|_{L^2} 
L_1^{\frac{1}{2}} \|g_1^{L_1,N_1}\|_{L^2} L_2^{\frac{1}{2}} 
\|g_2^{L_2,N_2}\|_{L^2} \\
& \lesssim & B N_1^{-\frac{1}{2}-s+} N_2^{-s+} N^{\sigma -1 +} 
L_1^{0+} L_2^{0+} L^{0+} \\
& \lesssim & B N_1^{-\frac{1}{2}-2s+} N^{\sigma -1 +} \, ,
\end{eqnarray*}
where
$$B:= N^{1-\sigma -} L^{\frac{1}{2}-} \|f^{L,N}\|_{L^2}  N_1^{s -} 
L_1^{\frac{1}{2}-} 
\|g_1^{L_1,N_1}\|_{L^2} N_2^{s -} L_2^{\frac{1}{2}-} \|g_2^{L_2,N_2}\|_{L^2} \, 
. $$
If $\sigma \le 1$ we get the bound $ \lesssim B N_1^{-\frac{1}{2}-2s+} \lesssim 
B$ , because $s > -\frac{1}{4}$ , whereas in the case $ \sigma > 1$ we estimate 
by $ \lesssim B N_1^{-\frac{3}{2}-2s+\sigma+} \lesssim B$ , because $\sigma -2s 
< \frac{3}{2}$ .
Dyadic summation over $L,L_1,L_2 \ge 1$ and $N,N_1,N_2$ gives the claimed 
result.\\
{\bf b.} In the case $N_1^2 < \max(L,L_1,L_2)$ we get by \cite[formula 
(3.25)]{BH} the same estimate as in a. with $N_1^{-\frac{1}{2}+}$ replaced by 
$N_1^{-\frac{1}{2}} N_1 \max(L,L_1,L_2)^{-\frac{1}{2}} \le N_1^{-\frac{1}{2}}$ , 
so that the same argument applies.\\
{\bf Case 2:} $ N_1 \ll N_2 $ ($ \Longrightarrow \, N \sim N_2$) (or $N_2 \ll 
N_1$ , which is the same problem).\\
{\bf a.} In the case $L_2 \gtrsim N_2^2 $ we get by \cite[formula (3.27)]{BH} :
\begin{eqnarray*}
\lefteqn{ |I(f^{L,N},g_1^{L_1,N_1},g_2^{L_2,N_2})| } \\
& \lesssim & N_1^{\frac{1}{2}} \min(L,L_1)^{\frac{1}{2}} \min(N_1^2, 
\max(L,L_1))^{\frac{1}{2}} \|f^{L,N}\|_{L^2} \|g_1^{L_1,N_1}\|_{L^2} 
\|g_2^{L_2,N_2}\|_{L^2} \\
& \lesssim & B N_1^{\frac{1}{2}-s+} N_2^{-s+} N^{\sigma -1+} L^{-\frac{1}{2}+} 
L_1^{-\frac{1}{2}+} L_2^{-\frac{1}{2}+} \min(L,L_1)^{\frac{1}{2}} 
\min(N_1^2,\max(L,L_1))^{\frac{1}{2}}  
\end{eqnarray*}
{\bf a1.} $L,L_1 \le N_1^2 $ \\
In this case we can estimate this by
$$
\lesssim B N_1^{\frac{1}{2}-s+} N_2^{-s+} N^{\sigma -1+}
L_2^{-\frac{1}{2}+} 
 \lesssim  B N_1^{\frac{1}{2}-s+} N^{\sigma -2-s+} \, . $$
If $s \le \frac{1}{2}$ we get the bound $\lesssim B N^{-\frac{3}{2}-2s+\sigma+} 
\lesssim B $ , because $\sigma -2s < \frac{3}{2}$ , whereas in the case $s > 
\frac{1}{2} $ we estimate by $ \lesssim B N^{\sigma -2-s+} \lesssim B$ , because 
$\sigma -2 < s$ .
Dyadic summation gives the desired bound.\\
{\bf a2.} $ \max(L,L_1) > N_1^2 $ \\ 
We estimate in this case by
$$
\lesssim  B N_1^{\frac{3}{2}-s+} N_2^{-s+} N^{\sigma -1+} L_2^{-\frac{1}{2}+} 
\max(L,L_1)^{-\frac{1}{2}+} 
 \lesssim  B N_1^{\frac{1}{2}-s+} N^{\sigma -2-s+} \, , $$
the same bound as in a1.\\
{\bf b.} In the case $L_2 \ll N_2^2 $ we get by \cite[formula (3.26) and 
(3.28)]{BH} :
$ \max(L,L_1,L_2) \gtrsim N_2^2 $ and
\begin{eqnarray*}
\lefteqn{ |I(f^{L,N},g_1^{L_1,N_1},g_2^{L_2,N_2})| } \\
& \lesssim & N_1 N_2^{-\frac{1}{2}} (L_1 L_2 L)^{\frac{1}{2}} 
\max(L,L_1,L_2)^{-\frac{1}{2}}  \|f^{L,N}\|_{L^2} \|g_1^{L_1,N_1}\|_{L^2} 
\|g_2^{L_2,N_2}\|_{L^2} \\
& \lesssim & B N_1^{1-s+} N_2^{-\frac{1}{2}-s+} N^{\sigma -1+} (L_1 L_2 L)^{0+} 
\max(L;L_1,L_2)^{-\frac{1}{2}} \\
& \lesssim &  B N_1^{1-s+} N^{-\frac{1}{2} -s +\sigma -2+} \\  
& \lesssim &  B N_1^{\frac{1}{2}-s+} N^{\sigma -2-s+} \, , 
\end{eqnarray*}
the same bound as in a1.\\
{\bf Case 3:} $ N \lesssim 1 $ ($\Longrightarrow N_1 \sim N_2$ or $N,N_1,N_2 
\lesssim 
1$.) \\
Assuming without loss of generality $L_1 \le L_2$ and using the bilinear 
Strichartz type estimate \cite[Prop. 4.3]{BH} we get
\begin{eqnarray*}
\lefteqn{ |I(f^{L,N},g_1^{L_1,N_1},g_2^{L_2,N_2})| } \\
&& \le  \| f^{L,N} g_1^{L_1,N_1}\|_{L^2} \|g_2^{L_2,N_2}\|_{L^2} \\
&& \lesssim \min(N,N_1) N_1^{-\frac{1}{2}} L^{\frac{1}{2}} L_1^{\frac{1}{2}} 
\|f^{L,N}\|_{L^2} \|g_1^{L_1,N_1}\|_{L^2} \|g_2^{L_2,N_2}\|_{L^2} \\
&& \lesssim  N_1^{-\frac{1}{2}} L^{\frac{1}{2}}  \|f^{L,N}\|_{L^2} 
L_1^{\frac{1}{4}} \|g_1^{L_1,N_1}\|_{L^2} L_2^{\frac{1}{4}} 
\|g_2^{L_2,N_2}\|_{L^2} \, .
\end{eqnarray*}
Furthermore we get by \cite[formula (4.22)]{BH}
$$ |I(f^{L,N},g_1^{L_1,N_1},g_2^{L_2,N_2})| \lesssim L^{\frac{1}{3}} 
\|f^{L,N}\|_{L^2} L_1^{\frac{1}{3}} \|g_1^{L_1,N_1}\|_{L^2} L_2^{\frac{1}{3}} 
\|g_2^{L_2,N_2}\|_{L^2} \, , $$
so that by interpolation we arrive at
\begin{eqnarray*}
 |I(f^{L,N},g_1^{L_1,N_1},g_2^{L_2,N_2})| & \lesssim & N_1^{-\frac{1}{2}+} 
L^{\frac{1}{2}-} \|f^{L,N}\|_{L^2} L_1^{\frac{1}{4}+} \|g_1^{L_1,N_1}\|_{L^2} 
L_2^{\frac{1}{4}+} \|g_2^{L_2,N_2}\|_{L^2} \\
& \lesssim & B N_1^{-\frac{1}{2}-s+} N_2^{-s+} N^{\sigma -1+} \\
& \lesssim & BN_1^{-\frac{1}{4}-s+} N_2^{-\frac{1}{4}-s+} \\
& \lesssim & B  
\end{eqnarray*}
using $s > - \frac{1}{4}$ .
Dyadic summation in all cases completes the proof of 
Theorem \ref{Theorem 5}.
\end{proof} 
\begin{proof} [Proof of Theorem \ref{Theorem 1}] It is by now standard to use 
(the remark 
to) Theorem \ref{Theorem 4} and Theorem 
\ref{Theorem 5} to show the local well-posedness result (Theorem \ref{Theorem 
1}) for the 
system (\ref{0.1'}),(\ref{0.2'}),(\ref{0.3'}) as an application of the 
contraction mapping principle. For details 
of the method we refer to \cite{GTV}. This solution then immediately leads to a 
solution of the Klein-Gordon-Schr\"odinger system 
(\ref{0.1}),(\ref{0.2}),(\ref{0.3}) with the required properties as explained 
before Theorem \ref{Theorem 1}.

Moreover, if $(u,n,\partial_t n)$ is a solution of (the system of integral 
equations belonging to) (\ref{0.1}),(\ref{0.1}),(\ref{0.3}) with $u \in 
X^{s,\frac{1}{2}+}[0,T]$ and data $u_0 \in H^s$, $n_0 \in H^{\sigma}$, $n_1 \in 
H^{\sigma}$, then $n_{\pm}$ defined by (\ref{9}) belongs to 
$X_{\pm}^{\sigma,\frac{1}{2}+}[0,T]$ by Theorem \ref{Theorem 5} and thus $n= 
\frac{1}{2}(n_+ + n_-)$ belongs to $X_+^{\sigma,\frac{1}{2}+}[0,T] + 
X_-^{\sigma,\frac{1}{2}+}[0,T]$ and 
$\partial_t n = \frac{1}{2i} A^{\frac{1}{2}}(n_+ - n_-)$ belongs to $ 
X_+^{\sigma -1,\frac{1}{2}+,}[0,T] + X_-^{\sigma -1,\frac{1}{2}+}[0,T]$ , and
one easily checks that $(u,n_+,n_-)$ is a solution of the system (of integral 
equations belonging to) 
(\ref{0.1'}),(\ref{0.2'}),(\ref{0.3'}). But because this solution is uniquely 
determined the solution of the Klein - Gordon - Schr\"odinger system is also 
unique.
\end{proof} 

\section{Unconditional uniqueness}
In this section we show that solutions of the KGS system are unique in its 
natural solution spaces in the important case, where the Cauchy data for the 
Schr\"odinger part belong to $L^2$ and the data for the Klein-Gordon-part 
belong to $L^2 \times H^{-1}$, namely in the space $C^0([0,T],L^2)$ and 
$C^0([0,T],L^2) \times C^0([0,T],H^{-1})$, respectively. This is of particular 
interest, because we know that in this case the solution exists globally in 
time by the result of \cite{CHT}.

First we show
\begin{prop}

\label{Prop. 1}
Let $u_0 \in L^2({\mathbb R}^3)$ , $n_0 \in L^2({\mathbb R}^3)$ , $n_1 \in 
H^{-1}({\bf 
R}^3)$ and $T > 0$ be given. Any solution
$$ (u,n_+,n_-) \in C^0([0,T],L^2({\mathbb R}^3)) \times  C^0([0,T],L^2({\mathbb 
R}^3)) 
\times C^0([0,T],L^2({\mathbb R}^3)) $$
of the system (\ref{0.1'}),(\ref{0.2'}),(\ref{0.3'}) belongs to 
$$ X^{0-,\frac{1}{2}+}[0,T] \times  X_+^{-\frac{1}{4}-,\frac{1}{2}+}[0,T] 
\times X_-^{-\frac{1}{4}-,\frac{1}{2}+}[0,T] \, . $$
\end{prop}
\begin{proof} [Proof of Theorem \ref{Theorem 3}] We again remark that any 
solution
$$ (u,n,\partial_t n) \in C^0([0,T],L^2({\mathbb R}^3)) \times  
C^0([0,T],L^2({\mathbb R}^3)) 
\times C^0([0,T],H^{-1}({\mathbb R}^3)) $$ of the Klein-Gordon-Schr\"odinger 
system (\ref{0.1}),(\ref{0.2}),(\ref{0.3}) leads to a corresponding solution of 
the system (\ref{0.1'}),(\ref{0.2'}),(\ref{0.3'}) with $$ (u,n_+,n_-) \in 
C^0([0,T],L^2({\mathbb R}^3)) \times  C^0([0,T],L^2({\mathbb 
R}^3)) 
\times C^0([0,T],L^2({\mathbb R}^3)) \, .$$ 
Thus combining Prop. \ref{Prop. 1} with the local well-posedness result Theorem 
\ref{Theorem 1} and 
the global existence result of \cite{CHT} we immediately get Theorem 
\ref{Theorem 3}.
\end{proof}
\begin{proof} [Proof of Prop. \ref{Prop. 1}] For the first part we use an idea 
of 
Y. Zhou 
\cite{Z},\cite{Z1}. By Sobolev's embedding theorem we get
\begin{eqnarray*}
\| n_{\pm} u\|_{L^2((0,T),H^{-\frac{3}{2}-})} & \lesssim & \|n_{\pm} u 
\|_{L^2((0,T),L^1)} \\
& \lesssim & T^{\frac{1}{2}} \|n_{\pm}\|_{L^{\infty}((0,T);L^2)} 
\|u\|_{L^{\infty}((0,T),L^2)} < \infty \, . 
\end{eqnarray*}
so that from (\ref{0.1'}) we have $ u \in X^{-\frac{3}{2}-,1}[0,T]$ , because
$$ \|(i\partial_t + \Delta)u\|^2_{L^2((0,T,H^{-\frac{3}{2}-})} + 
\|u\|^2_{L^2((0,T),H^{-\frac{3}{2}-})} \sim \|u\|^2_{X^{-\frac{3}{2}-,1}[0,T]} 
< \infty \, . $$
Similarly we get
$$ \| |u|^2 \|_{L^2((0,T),H^{-\frac{3}{2}-})} \lesssim T^{\frac{1}{2}} 
\|u\|^2_{L^{\infty}((0,T),L^2)} < \infty $$
and therefore $ A^{-1/2} (|u|^2) \in L^2((0,T),H^{-\frac{1}{2}-}) $ . From 
(\ref{0.2'}) we conclude $n_{\pm} \in X_{\pm}^{-\frac{1}{2}-,1}[0,T]$ .\\
Interpolation with $u \in X^{0,0}[0,T]$ and $ n_{\pm} \in X_{\pm}^{0,0}[0,T]$ 
gives 
$ u \in X^{-\frac{3}{2}\Theta -, \Theta}[0,T],$  especially $ u \in 
X^{-\frac{3}{4}-\frac{3}{2}\epsilon-,\frac{1}{2}+\epsilon}[0,T] $ , 
and 
$ n_{\pm} \in X_{\pm}^{-\frac{1}{2}\Theta -,\Theta}[0,T]$ , especially $n_{\pm} 
\in X_{\pm}^{-\frac{1}{4}-,\frac{1}{2}+}[0,T]$ 
for any $ 0 \le \Theta \le 1 $ . \\
We now improve the regularity of $u$ keeping the regularity of $n$ fixed. We 
suppose $ u \in X^{-\frac{3}{4} (\frac{4}{5})^n - , \frac{1}{2}+}$ and $ 
n_{\pm} \in X_{\pm}^{-\frac{1}{4}-,\frac{1}{2}+} $ , which is fulfilled for 
$n=0$ , and want to conclude $ u \in 
X^{-\frac{3}{4}(\frac{4}{5})^n-,\frac{5}{8}-} $ for all $n \in {\mathbb N} \cup 
\{0\}$ . Assuming this for the moment we interpolate this result with $ u \in 
X^{0,0}$ with interpolation parameter $\Theta = \frac{4}{5}+$ and conclude $ u 
\in X^{-\frac{3}{4}(\frac{4}{5})^{n+1}-,\frac{1}{2}+}$ , so that the iteration 
works and finally gives $ u \in X^{0-,\frac{1}{2}+}[0,T] $ . Recalling $n_{\pm} 
\in X_{\pm}^{-\frac{1}{4}-,\frac{1}{2}+}[0,T] $ the uniqueness part of our 
local well-posedness result Theorem \ref{Theorem 1} gives the claimed result. 
So we are done if we prove
\begin{equation}
\label{**}
u \in X^{s,\frac{1}{2}+\epsilon} \, , \, n_{\pm} \in 
X_{\pm}^{-\frac{1}{4}-,\frac{1}{2}+} \quad \Longrightarrow \quad u \in 
X^{s,\frac{5}{8}-2\epsilon}
\end{equation}
for $ -\frac{3}{4}-\frac{3}{2}\epsilon - \le s < 0 $ and any $\epsilon > 0$ .\\
This means that we have to show
$$ \|u n_{\pm} \|_{X^{s,-\frac{3}{8}-\epsilon}} \lesssim 
\|u\|_{X^{s,\frac{1}{2}+\epsilon}} 
\|n_{\pm}\|_{X_{\pm}^{-\frac{1}{4}-,\frac{1}{2}+}} \, . $$
This is a consequence of Theorem \ref{Theorem 4}, where we choose the parameters 
as follows:\\
$b_1 = \frac{3}{8}+\epsilon$ , $b_2 = \frac{1}{2}+\epsilon$ , $ b = \frac{1}{2}$ 
, $\sigma = -\frac{1}{4}-$ .\\
Then one easily checks that the conditions (\ref{10}),(\ref{11}) and (\ref{12}) 
are satisfied, provided $ s \ge -\frac{3}{4}-\frac{3}{2}\epsilon -$ .

This completes the proof of Prop. 
\ref{Prop. 1}, and thus Theorem \ref{Theorem 3} is also proven.
\end{proof}

\section{ Sharpness of the well-posedness result}
In this section we show that the local well-posedness result is sharp up to the 
endpoints. First we construct counterexamples which show that the threshold 
on the parameters $s$ and $\sigma$ in Theorem \ref{Theorem 4} and 
Theorem \ref{Theorem 5} is essentially necessary. We follow the arguments of 
\cite{H} and \cite{BHHT}.
\begin{prop}
\label{Prop. 2}
Assume $ s \in {\mathbb R}$ , $b',b_1,b_2 \ge 0$ and $ {\sigma} < - \frac{1}{2}$ 
. Then 
the 
inequality
$$ \|uv\|_{X^{s,-b'}} \lesssim \|v\|_{X^{{\sigma},b_1}_{\pm}} \|u\|_{X^{s,b_2}} 
$$
is false.
\end{prop}
\begin{proof} \cite[Prop. 6.1]{BHHT}. The two dimensional argument 
carries over 
to three dimensions.
\end{proof}
\begin{prop}
\label{Prop. 3}

(a) Assume ${\sigma},s \in {\mathbb R}$ , $b',b_1,b_2 \ge 0$ and $ 
{\sigma}-2s-\frac{3}{2} > 
0$. 
Then the inequality
$$ \|A^{-1/2}(u\overline{w})\|_{X_{\pm}^{{\sigma},-b'}} \lesssim 
\|u\|_{X^{s,b_1}} 
\|w\|_{X^{s,b_2}} $$
is false.\\
(b) The same holds true, if $ s < - \frac{1}{4}$ .

\end{prop}
\begin{proof} 
(a) \cite[Prop. 6.2]{BHHT}. Replace ${\sigma}$ by ${\sigma}+2$  and use their 
two-dimensional argument.\\
(b) follows immediately from the following
\renewcommand{\qedsymbol}{}

\end{proof}
\begin{lemma}
\label{Lemma 1}
Assume ${\sigma},s \in {\mathbb R}$ and $b',b_1,b_2 \ge 0$. For any $N \gg 1$ 
there 
exist 
functions $u_N$ and $w_N$ and a constant $c_0 > 0$ independent of $N$ such that
$$ \frac{ \|A^{-1/2}(u_N 
\overline{w}_N)\|_{X_{\pm}^{{\sigma},-b'}}}{\|u_N\|_{X^{s,b_1}} 
\|w_N\|_{X^{s,b_2}}} 
\ge c_0 N^{-2s-\frac{1}{2}} \, . $$
\end{lemma}
\begin{proof} Let $\widehat{u}_N = \chi_E $ (= characteristic function 
of the 
set $E$), where
\begin{eqnarray*}
E & = & \{ (\xi_1,\xi_2,\xi_3,\tau) \in {\mathbb R}^4 : N - \frac{1}{N} \le 
\xi_1 
\le N + \frac{1}{N} , -1 \le \xi_2 , \xi_3 \le 1 , \\
& & \hspace{15em} -N^2-1 \le \tau \le -N^2+1 \} \, , 
\end{eqnarray*}
Moreover let
$\widehat{w}_N = \chi_F $ , where
\begin{eqnarray*}
F & = & \{ - N + \frac{3}{N} \le \xi_1 \le -N + \frac{5}{N} , -1 \le \xi_2 , 
\xi_3 \le 1 , \\
& & \hspace{5em} -(N-\frac{4}{N})^2-1 \le \tau \le -(N-\frac{4}{N})^2+1 \} \, , 
\end{eqnarray*}
so that $\langle \tau + |\xi|^2 \rangle \sim 1 $ and $\langle \xi \rangle \sim 
N$ on $E$ and $F$. This implies $ \widehat{u_N \overline{w}_N}(\xi,\tau) \gtrsim 
N^{-1}$ on a rectangle $G$ centered around points $\xi_1^0 = 0(\frac{1}{N})$ , 
$ \xi_2^0 = \xi_3^0 = 0$ , $\tau^0 = 0(1)$ of width $0(\frac{1}{N}) , 0(1), 
0(1)$ and $0(1)$, respectively, so that $ \langle \tau + |\xi|^2 \rangle \sim 1 
$ and $ \langle \xi \rangle \sim 1 $ on $G$. Consequently
$$
\| A^{-1/2} (u_N \overline{w}_N) \|_{X_{\pm}^{{\sigma},-b'}}
\gtrsim ( \int_G ( \langle \xi \rangle^{{\sigma}-1}
\widehat{u_N \overline{w}_N}(\xi,\tau) \langle |\xi|^2 + \tau \rangle^{-b'})^2 
d\xi d\tau)^{\frac{1}{2}} \gtrsim N^{-\frac{3}{2}} \, ,
$$
whereas
$$ \|u_N\|_{X^{s,b_1}} + \|w_N\|_{X^{s,b_2}} \lesssim N^{s-\frac{1}{2}} \, . $$
so that the proof is complete.
\end{proof}

The following Propositions \ref{Prop. 4}, \ref{Prop. 5} and \ref{Prop. 6} show 
that the flow map of our Cauchy problem is not $C^2$ so that the problem is 
ill-posed in this sense. This strategy of proof goes back to Bourgain \cite{B}, 
Tzvetkov \cite{T} and Molinet-Saut-Tzvetkov \cite{MST},\cite{MST1} and is taken 
up by Holmer \cite{H}. Thus the proof of Theorem \ref{Theorem 2} immediately 
follows from 
these propositions by the arguments of Holmer \cite{H}.
\begin{prop}
\label{Prop. 4}
Let $0 < T \le 1$. For any $N\gg 1$ there exists $u_N \in H^s({\mathbb R}^3)$ 
and a 
constant $c_0 > 0$, which is independent of $N$ such that
$$
\sup_{|t| \le T} \Bigl\| \int_0^t e^{-i(t-t')A^{1/2}} A^{-1/2}(e^{it'\Delta} 
u_N \overline{e^{it'\Delta}u_N}) dt' \Bigr\|_{H^{\sigma}({\mathbb R}^3)} \ge c_0 
N^{{\sigma}-2s-\frac{3}{2}} \|u_N\|^2_{H^s({\mathbb R}^3)}.
$$
\end{prop}
\begin{proof} Let $\widehat{u}_N := \chi_{D_1} + \chi_{D_2}$, where
\begin{eqnarray*}
D_1 & := & \{ \xi \in {\mathbb R}^3 : N+1-N^{-1} \le \xi_1 \le N+1+N^{-1} , -1 
\le 
\xi_2 , \xi_3 \le 1 \} \, ,\\
D_2 & := & \{ \xi \in {\mathbb R}^3 : -N-2N^{-1} \le \xi_1 \le -N+2N^{-1} , -2 
\le 
\xi_2 , \xi_3 \le 2 \} \, .
\end{eqnarray*}
In order to treat $\widehat{u_N \overline{u}_N}(\xi)$ one has to consider 4 
terms. We have
\begin{equation}
\label{***}
 \Bigl| {\mathcal F}_x ( \int_0^t e^{i(t-t')A^{1/2}} A^{-1/2} (e^{it'\Delta} 
\check{\chi}_{D_1} \overline{e^{it'\Delta} \check{\chi}_{D_2}}) dt') (\xi) 
\Bigr| 
 \end{equation}
$$
 \sim  \langle \xi \rangle ^{-1} \Bigl| \int_0^t \int_{\eta \in D_1,\eta - \xi 
\in D_2} e^{it'(\langle \xi \rangle - |\eta|^2 + |\xi - \eta|^2)} 
\chi_{D_1}(\eta) \chi_{D_2}(\eta - \xi) d \eta dt' \Bigr| \, .
$$
The inner integral vanishes unless
$$ \xi \in D':=\{2N+1-3N^{-1} \le \xi_1 \le 2N+1+3N^{-1} , -3 \le \xi_2,\xi_3 
\le 3 \} \, . $$
Now for the phase factor we have for such $\xi$:
$$ \langle \xi \rangle - |\eta|^2 + |\xi - \eta|^2 = 2N+1-(N+1)^2+N^2 + 0(1) = 
0(1) \, . $$
so that for $|t| \ll 1$ we get $|t'(\langle \xi \rangle - 
|\eta|^2+|\xi-\eta|^2)| \ll 1 $ and $ \langle \xi \rangle \sim 2N $, and 
therefore
$$ \int_0^t e^{it'(\langle \xi \rangle - |\eta|^2 + |\xi - \eta|^2)} dt' \sim t 
\, . $$
Moreover, if $\eta \in D_1$ and $\xi \in D$, where
$$ D:= \{ 2N+1-N^{-1} \le \xi_1 \le 2N+1+N^{-1} , -1 \le \xi_2,\xi_3 \le 1 \} 
\, $$
then automatically $\eta - \xi \in D_2$, so that for such $\xi$ the region of 
integration over $\eta$ is of size $|D_1| \sim N^{-1}$ . Thus for $\xi \in D$
we get:  $ \quad (\ref{***})  \gtrsim |t| N^{-2} \quad $ , so that integration 
over $\xi \in D$ with $|D| \sim N^{-1}$ gives
\begin{equation}
\label{****}
\Bigl\| \int_0^t e^{i(t-t')A^{1/2}} A^{-1/2} (e^{it'\Delta} \check{\chi}_{D_1} 
\overline{e^{it'\Delta} \check{\chi}_{D_2}}) dt')\Bigr\|_{H^{\sigma}({\mathbb 
R}^3)} 
\gtrsim |t| N^{{\sigma}-\frac{5}{2}} \, .
\end{equation}  
Next we treat the term where the roles of $D_1$ and $D_2$ are exchanged. It 
vanishes unless
$$ \xi \in \{ -2N-1-3N^{-1} \le \xi_1 \le -2N-1+3N^{-1} , -3 \le \xi_2,\xi_3 
\le 3 \} \, , $$
so that its support is disjoint to the support $D'$ in the previous case. 
Similarly the term coming from the product of $ e^{it'\Delta} \check 
{\chi}_{D_1} \overline{e^{it'\Delta}\check{\chi}_{D_1}} $ and $ e^{it'\Delta} 
\check {\chi}_{D_2} \overline{e^{it'\Delta}\check{\chi}_{D_2}} $ vanishes 
unless $\xi \in \{ -2N^{-1} \le \xi_1 \le 2N^{-1} , -2 \le \xi_2,\xi_3 \le 2 
\}$ and $\xi \in \{ -4N^{-1} \le \xi_1 \le 4N^{-1} , -4 \le \xi_2,\xi_3 \le 4 
\}$, respectively, so that these supports are also disjoint to $D'$. This 
implies
$$ \Bigl\| \int_0^t e^{-i(t-t')A^{1/2}} A^{-1/2}(e^{it'\Delta} u_N 
\overline{e^{it'\Delta}u_N}) dt' \Bigr\|_{H^{\sigma}({\mathbb R}^3)} \gtrsim |t| 
N^{{\sigma}-\frac{5}{2}} \, . $$
Combined with $ \|u_N\|_{H^s} \sim N^{s-\frac{1}{2}} $ this gives the claimed 
result\end{proof}
\begin{prop}
\label{Prop. 5}
Let $0<T \le 1$. For any $N \gg 1$ there exists $u_N \in H^s({\mathbb R}^3) $ 
such 
that
$$
\sup_{|t| \le T} \Bigl\| \int_0^t e^{-i(t-t')A^{1/2}} A^{-1/2}(e^{it'\Delta} 
u_N \overline{e^{it'\Delta}u_N}) dt' \Bigr\|_{H^l({\mathbb R}^3)} \ge c_0 
N^{-2s-\frac{1}{2}} \|u_N\|^2_{H^s({\mathbb R}^3)} $$
for any ${\sigma} \in {\mathbb R}$, where $c_0 > 0$ is independent of $N$.
\end{prop}
\begin{proof} Let $\widehat{u}_N := \chi_{D_1} + \chi_{D_2}$, where
\begin{eqnarray*}
D_1 & := & \{ \xi \in {\mathbb R}^3 : N-N^{-1} \le \xi_1 \le N+N^{-1} , -1 \le 
\xi_2 , \xi_3 \le 1 \} \, ,\\
D_2 & := & \{ \xi \in {\mathbb R}^3 : N+4N^{-1} \le \xi_1 \le N+7N^{-1} , -2 \le 
\xi_2 , \xi_3 \le 2 \} \, .
\end{eqnarray*}
We first consider the term (\ref{***}). This vanishes unless
$\xi \in D'$, where 
$$ D' := \{ -8N^{-1} \le \xi_1 \le -3N^{-1} , -3 \le \xi_2,\xi_3 \le 3 \}  \, , 
$$ so that for $\xi \in D'$ we have $| \langle \xi \rangle - |\eta|^2 + |\xi - 
\eta|^2| \lesssim 1 $ and $ \langle \xi \rangle \sim 1$, thus for $|t| \ll 1$ :
$$ \int_0^t e^{it'(\langle \xi \rangle - |\eta|^2 + |\xi - \eta|^2)} dt' \sim t 
\, . $$
Moreover, if $\eta \in D_1$ and $\xi \in D$, then automatically $\eta - \xi \in 
D_2$, where 
$$D:= \{ -6N^{-1} \le \xi_1 \le -5N^{-1} , -1 \le \xi_2,\xi_3 \le 1 \} \, , $$ 
so that for such $\xi$ the region for the integration over $\eta$ is of size 
$|D_1| \sim N^{-1}$, which implies the lower bound: 
\begin{equation}
 \Bigl| {\mathcal F}_x ( \int_0^t e^{i(t-t')A^{1/2}} A^{-1/2} (e^{it'\Delta} 
\check{\chi}_{D_1} \overline{e^{it'\Delta} \check{\chi}_{D_2}}) dt') (\xi) 
\Bigr|  
 \gtrsim  |t| N^{-1}  \, , 
\end{equation} 
 so that integration over $\xi \in D$ with $|D| \sim N^{-1}$ and $ \langle \xi 
\rangle \sim 1 $ gives:
$$
\Bigl\| \int_0^t e^{i(t-t')A^{1/2}} A^{-1/2} (e^{it'\Delta} \check{\chi}_{D_1} 
\overline{e^{it'\Delta} \check{\chi}_{D_2}}) dt')\Bigr\|_{H^l({\mathbb R}^3)} 
\gtrsim |t| N^{-\frac{3}{2}} \, .
$$
Next we consider the integrals
$$ \int e^{it'(\langle \xi \rangle - |\eta|^2 
+ |\xi - \eta|^2)} \chi_{D_j}(\eta) \chi_{D_m}(\eta - \xi) d \eta $$
for the other combinations of $ j,m \in \{1,2\} $. This term vanishes unless
\begin{itemize}
\item in the case $j=2,m=1$: $\, \xi \in \{ 3N^{-1} \le \xi_1 \le 8N^{-1} , -3 
\le \xi_2,\xi_3 \le 3 \} \, , $
\item in the case $j=m=1$: $\, \xi \in \{ -2N^{-1} \le \xi_1 \le 2N^{-1} , -2 
\le \xi_2,\xi_3 \le 2 \} \, , $
\item in the case $j=m=2$: $\, \xi \in \{ -3N^{-1} \le \xi_1 \le 3N^{-1} , -4 
\le \xi_2,\xi_3 \le 4 \} \, . $
\end{itemize}
In all the cases these sets are disjoint to $D'$, so that we conclude
$$ \Bigl\| \int_0^t e^{-i(t-t')A^{1/2}} A^{-1/2}(e^{it'\Delta} u_N 
\overline{e^{it'\Delta}u_N}) dt' \Bigr\|_{H^{\sigma}({\mathbb R}^3)} \ge c_0 t 
N^{-\frac{3}{2}} \, . $$
Combined with $ \|u_N\|_{H^s} \sim N^{s-\frac{1}{2}} $ this gives the claimed 
result.
\end{proof}
\begin{prop}
\label{Prop. 6}
Let $0<T\le 1$. For all $ N \gg T^{-1}$ there exists $u_N \in H^s({\mathbb 
R}^3)$ 
and $v_N \in H^{\sigma}({\mathbb R}^3)$ such that
\begin{eqnarray*}
\lefteqn{\sup_{|t| \le T} \Bigl\| \int_0^t e^{-i(t-t')\Delta} (e^{it'\Delta} 
u_N Re(e^{-it'A^{1/2}}v_N)) dt' \Bigr\|_{H^s({\mathbb R}^3)} }\\
& & \ge c_0 N^{-{\sigma}-\frac{1}{2}} \|u_N\|_{H^s({\mathbb R}^3)} 
\|v_N\|_{H^{\sigma}({\bf 
R}^3)} 
\end{eqnarray*}
for any $s \in {\mathbb R}$ , where $c_0 > 0$ is independent of $N$.
\end{prop}
\begin{proof} The proof of \cite[Prop. 6.5]{BHHT} in two dimensions is 
also true 
in the three dimensional case with obvious modifications.
\end{proof}

\end{document}